\documentclass[11pt]{article}

\usepackage{amsmath, amsfonts, amssymb, amsthm}

\usepackage{natbib}

\usepackage{setspace}

\usepackage{setspace}

\DeclareMathSymbol{\leqslant}{\mathalpha}{AMSa}{"36} 

\DeclareMathSymbol{\geqslant}{\mathalpha}{AMSa}{"3E} 

\DeclareMathSymbol{\eset}{\mathalpha}{AMSb}{"3F}     

\newcommand{\cE}{\mathcal{E}}

\newcommand{\Om}{\Omega}

\newcommand{\om}{\omega}

\newcommand{\eps}{\varepsilon}

\newcommand{\ban}{\begin{align}}
\newcommand{\ean}{\end{align}}

\newcommand{\ba}{\begin{align*}}
\newcommand{\ea}{\end{align*}}

\newcommand{\be}{\begin{eqnarray*}}
\newcommand{\ee}{\mathrm{e}}

\newcommand{\ben}{\begin{eqnarray}}
\newcommand{\een}{\end{eqnarray}}

\theoremstyle{plain}
\newtheorem{theo}{Theorem}[section]
\newtheorem{lemma}[theo]{Lemma}
\newtheorem{propo}[theo]{Proposition}

\newtheorem{defi}[theo]{Definition}

\theoremstyle{definition}

\newcommand{\R}{{\mathbb R}}

\renewenvironment{proof}[1][] {\noindent{\bf Proof#1.} }{\hspace*{\fill}$\square$\medskip\par}

\begin{document}

\vglue20pt \centerline{\huge\bf Attractors and Expansion}

\medskip

\centerline{\huge\bf  for Brownian Flows}

 \medskip

\bigskip

\bigskip

\centerline{by}

\bigskip

\medskip

\centerline{{\Large  G.\ Dimitroff\footnotemark[1] and  M.\ Scheutzow\footnotemark[2]}}
\footnotetext[1]{Fraunhofer ITWM, Fraunhofer-Platz 1, D-67663 Kaiserslautern}

\footnotetext[2]{Institut f\"ur Mathematik, MA 7-5, Technische Universit\"at Berlin, Stra\ss e des 17. Juni 136, D-10623 Berlin }

\bigskip

\bigskip

\bigskip

{\leftskip=1truecm

\rightskip=1truecm

\baselineskip=15pt

\small

\noindent{\slshape\bfseries Summary.}  We show that a stochastic flow which is generated by a stochastic differential 
equation on $\R^d$ with 
bounded volatility has a random attractor provided that the drift component in the 
direction towards the origin is larger than a certain strictly positive constant $\beta$ outside a large ball. 
Using a similar approach, we provide a lower bound for the linear growth rate of the inner radius of the image of a large 
ball under a stochastic flow in case the drift component in the 
direction away from the origin is larger than a certain strictly positive constant $\beta$ outside a large ball. 
To prove the main result we use {\em chaining techniques} in order to control the growth of the diameter of subsets of the 
state space under the flow.
\bigskip


\bigskip


}


\newcommand{\oq}{{\langle}}
\newcommand{\ve}{{\varepsilon}}
\newcommand{\cq}{{\rangle}_t}
\newcommand{\dd}{\,{\mathrm{d}}}
\newcommand{\Dd}{{\mathrm{D}}}
\newcommand{\cd}{{\cdot}}
\newcommand{\nix}{{\varnothing}}
\newcommand{\N}{{\mathbb N}}
\newcommand{\Z}{{\mathbb Z}}
\newcommand{\Q}{{\mathbb Q}}
\newcommand{\E}{{\mathbb E}}
\renewcommand{\P}{{\mathbb P}}
\newcommand{\F}{{\cal F}}
\newcommand{\C}{{\mathbb C}}
\newcommand{\K}{{\mathbb K}}
\newcommand{\B}{{\cal B}}
\newcommand{\G}{{\cal G}}
\newcommand{\D}{{\cal D}}
\newcommand{\Zz}{{\cal Z}}
\newcommand{\Ll}{{\cal L}}
\newcommand{\A}{{\cal A}}
\newcommand{\Po}{{\cal P}}
\newcommand{\Sy}{{\cal S}}
\newcommand{\cZ}{{\cal Z}}
\newcommand{\M}{{\cal M}}
\newcommand{\Nn}{{\cal N}}
\newcommand{\p}{{\mathbf P}}
\newcommand{\X}{{\mathbb X}}
\newcommand{\interior}[1]{{\mathaccent 23 #1}}
\newcommand{\bem}{\begin{em}}
\newcommand{\eem}{\end{em}}
\def\eins{{\mathchoice {1\mskip-4mu\mathrm l} 
{1\mskip-4mu\mathrm l}
{1\mskip-4.5mu\mathrm l} {1\mskip-5mu\mathrm l}}}
\newcommand{\olim}[1]{\begin{array}{c} ~\\[-1.4ex] \overline{\lim} \\[-1.35ex]
                       {\scriptstyle #1}\end{array}}
\newcommand{\plim}[2]{\begin{array}{c} #1\\[-1.175ex] \longrightarrow \\[-1.2ex] 
                       {\scriptstyle #2}\end{array}} 

\newcommand{\ulim}[1]{\begin{array}{c} ~\\[-1.175ex] \underline{\lim} \\[-1.2ex]
                       {\scriptstyle #1}\end{array}}

\renewcommand{\theequation}{\thesection.\arabic{equation}}

\section{Introduction}\label{intro}
It has been suggested that stochastic flows can be used as a model for studying the spread of passive tracers 
within a turbulent fluid. The individual particles (one-point motions) perform diffusions while the motions of 
adjacent particles are correlated in order to form a stochastic flow of homeomorphisms. Infinitesimally the 
flow is governed by a stochastic field of continuous semimartingales $F(t,x)$ via a stochastic differential 
equation (SDE) of {\em Kunita-type}
\begin{align*}
\phi_{s,t}(x)=x+\int\limits_s^t F(\dd u, \phi_{s,u}(x)).
\end{align*}
We will be interested in questions concerning the asymptotics of these flows. We aim at conditions 
for the existence of a {\em random} ({\em pullback}) {\em attractor}. 

One might expect that an SDE with bounded and Lipschitz diffusion coefficient and Lipschitz continuous 
drift $b$ whose component in the direction of the origin is positive and bounded away from zero outside a large ball 
should have a random attractor, i.e.\ large balls should contract under the solution (semi-)flow generated by 
the SDE and 
converge (in an appropriate sense) towards a stationary process taking values in the space of compact subsets of the 
state space $\R^d$ of the SDE. Interestingly, this need not be the case when $d \ge 2$ (\cite{S08B} contains a counterexample). 
It is clear that under the conditions above, the drift is sufficiently strong to push individual trajectories towards 
the origin when they are far away but it may happen that the drift is not sufficiently strong to push {\em all} 
trajectories starting far away towards the origin. In fact it may happen that a non-empty set of initial conditions 
(depending on the future 
of the driving noise process) will move towards infinity against the drift (with linear speed). Our main result will show 
however that there exists a number $\beta_0$ which is strictly positive in case $d \ge 2$ and zero for $d=1$ and which 
depends on the parameters of the SDE such that if the component of the drift $b$ in the direction of the origin is larger than 
$\beta_0$ outside a large ball, then {\em all} trajectories are attracted towards the origin and a random attractor (in the 
pullback sense) exists. 

Our main result, Theorem \ref{main}, contains a second statement which is in some sense dual to the first: if the 
component $\beta$ of the drift $b$ in the direction {\em away} from the origin is larger than $\beta_0$, then large balls are 
very likely to {\em expand} in the sense that each compact set will eventually be contained in (or swallowed by) 
the image of the ball under the flow meaning that the probability that this does not happen decreases to zero as the radius 
of the initial ball goes to infinity. In fact the results are even stronger: the speed of expansion is (at least) {\em linear} 
with rate at least $\beta -\beta_0$. Similarly, we show that the speed of attraction in the first result is also at least linear.

To prove the results, we provide bounds on the one-point motion of the solutions (these are fairly standard) and also 
estimates for the two-point motion (which are not as standard) which are needed to apply {\em chaining methods} in order to 
control the growth of sets (often small balls) under the action of the solution flow. We will be more explicit about our 
strategy in Section \ref{Proofs}. 

The paper is organized as follows: in the next section we provide the set-up and some basic 
definitions. In Section \ref{Mainresult} we state the main result. Section \ref{Proofs}  contains the proof. In the 
Appendix, we collect some auxiliary results.

\section{Set-up and preliminaries}

Let $(\Omega,\F, (\F_t)_{t \ge 0}, \P)$ be a filtered probability space satisfying the usual conditions. On this probability 
space we define a jointly continuous martingale field $M(t,x,\omega): \R_+\times \R^d\times \Omega \to \R^d$, 
satisfying $M(0,x)=0$ for all $x \in \R^d$. We assume that its joint quadratic 
variation is of the form
$$
\langle M(\cdot,x),M(\cdot,y)\rangle_t=t\cdot a(x,y)\\
$$
for a (deterministic) function $a :  \R^d\times \R^d \to \R^{d\times d}$. Note that this implies that 
$(t,x)\mapsto M(t,x)$ is a Gaussian field and $t \mapsto M(t,x)$ is a Brownian motion (up to a linear transformation) for each 
$x$. Further, we let $b :  \R^d \to \R^{d}$ be a (drift) vector field.
We consider a stochastic flow generated via a stochastic differential equation (SDE) of {\em Kunita type}, see \cite{Ku90}
\begin{align}\label{gen_sde}
\phi_{s,t}(x)=x+\int_s^t b(\phi_{s,u}(x))\dd u + \int_s^t M(\dd u, \phi_{s,u}(x)),\,0 \le s \le t < \infty.
\end{align}
We abbreviate $\A(x,y):=a(x,x)-a(x,y)-a(y,x)+a(y,y)$. Observe that 
$$
\A(x,y)=\frac{\dd}{\dd t} \langle M(\cdot,x)-M(\cdot,y)\rangle_t.
$$ We impose the following Lipschitz-type condition: \\[0.3cm]
\textbf{Condition (A1)} \textit{There are constants $\lambda \ge 0$ and $\sigma_L>0$ such that for all $x,y \in \R^d$ we have  
\begin{enumerate}
\item[\textnormal{(i)}]\,$||\A(x,y)||\le \sigma_L^2\, |x-y|^2$
\item[\textnormal{(ii)}]\,$b$ satisfies a local Lipschitz condition and 
$(x-y)\cdot (b(x)-b(y)) + \frac{d-1}{2}\,\sigma_L^2\,  |x-y|^2 \le \lambda |x-y|^2$.
\end{enumerate}}
Here $|\cdot|$ denotes the Euclidean norm in $\R^d$ and $||\cdot||$ the  operator norm for $d\times d$ matrices. \\
It is essentially well-known, that under assumptions (A1), the  SDE (\ref{gen_sde}) has a unique 
solution for each $x$ and $s$. Indeed this follows from a straightforward modification of Theorem 3.4.6 in 
\cite{Ku90} (this theorem requires a linear growth condition of the form $|b(x)\cdot x| \le c (1+|x|^2)$  but 
really only uses an estimate of the form $b(x) \cdot x \le c (1+|x|^2)$ which is an easy consequence of Condition (A1)).
Further, Theorem 4.7.1 in \cite{Ku90} shows that equation \eqref{gen_sde} generates a stochastic flow of local homeomorphisms 
(defined in \cite{Ku90}, p.177). Theorem \ref{ball} together with Lemma \ref{two} shows that (a modification of) this flow is 
actually {\em strongly complete} (or {\em strictly complete}) 
i.~e.\ 
$\phi: \left\{ 0\le s \le t < \infty \right\} \times \R^d \times \Om \to \R^d$ satisfies for each $\om \in \Om$
\begin{enumerate}
\item[\textnormal{(i)}] $\phi_{s,s}(\om)=\mathrm{id_{\R^d}}, \;\; s \ge 0$,
\item[\textnormal{(ii)}]  $(s,t,x) \mapsto \phi_{s,t}(x,\om)$ is continuous,
\item[\textnormal{(iii)}] for each $s,t$, the map $x \mapsto \phi_{s,t}(x,\om)$ is one-to-one, 
\item[\textnormal{(iv)}] $\phi_{s,u}(\om)=\phi_{t,u}(\om)\circ \phi_{s,t}(\om) \;\; u \ge t \ge s \ge 0$.
\end{enumerate} 
Additionally, $\phi$ has stationary and independent increments and is therefore called a {\em (time-homogeneous) Brownian flow}.
$\phi$ can be uniquely (in law) extended to $-\infty < s \le t < \infty$ in such a way that stationarity, independent increments and 
properties (i), (ii), (iii) and (iv) are preserved. In this case we will call $\phi$ the {\em flow generated by the SDE} \eqref{gen_sde}. 
Note that -- in general --  $\phi_{s,t}(\om)$ is not onto (not even 
in the deterministic case). Assumption (A1) allows for example for a drift $b(x)=-|x|^2\,x$ whose solution flow is not onto.  

Observe that flows generated via SDEs driven by finitely many independent Brownian motions 
\begin{align*}
\phi_{s,t}(x)=x+ \int_s^t V_0(\phi_{s,u}(x))\dd u + \sum_{i=1}^m \int_s^t V_i(\phi_{s,u}(x))\dd W_i(u),
\end{align*}
where $V_i:\R^d\to \R^d$ are globally Lipschitz vector fields, satisfy Assumption (A1) by  
taking $b(x)=V_0(x)$ and $M(t,x)=\sum_{i=1}^m V_i(x)W_i(t)$.\\

We will now formulate additional conditions which we will use in our main result.\\[0.3cm]
\textbf{Condition (A2)} \textit{ The diffusion coefficient is uniformly bounded, i.e.\ there exists $\sigma_B >0$ such that  
$||a(x,x)||\le \sigma_B^2$ for all $x\in \R^d$.}\\

For a given value of $\beta \in \R$, the following conditions require that the component of the drift in the radial direction 
is asymptotically bounded from above respectively below by $\beta$.\\

\noindent \textbf{Condition (A3$^{\beta}$)} 
$$
\limsup_{|x| \to \infty}\frac{x}{|x|}\cdot b(x) \le \beta.
$$

\noindent \textbf{Condition (A3$_{\beta}$)} 
$$
\liminf_{|x| \to \infty}\frac{x}{|x|}\cdot b(x) \ge \beta.
$$

\subsection{Attractors}

In this section we give a brief introduction to the concept of a random attractor. 
Let $(E,d)$ be a Polish (i.e.~a separable complete metric) space and let 
$\cE$ be its Borel $\sigma$-algebra. 
\begin{defi}
\begin{enumerate}
\item[\textnormal{(a)}]  $\left ( \Omega, \F^0, \P, \left (\theta_t \right)_{t \in \R}\right) $ is called a
\bem metric dynamical system (MDS)\eem, if  $\left ( \Omega, \F^0, \P\right)$ is a probability space and 
the family of mappings $\left\{\theta_t \colon \Omega \to \Omega \,| \, t \in \R \right \}$ satisfies
\begin{itemize}
\item[\textnormal{(i)}] the mapping $(\omega,t) \mapsto \theta_t(\omega)$ is $(\F^0 \otimes \B(\R),\F^0)$- measurable,
\item[\textnormal{(ii)}]$\theta_0=\textnormal{id}_{\Omega}$,
\item[\textnormal{(iii)}] \bem(flow-property) \eem  $\theta_{s+t}=\theta_s\circ \theta_t$ for all $s,t \in \R$,
\item[\textnormal{(iv)}] for each $t \in \R$, $\theta_t$ preserves the measure $\P$.
\end{itemize}
\item[\textnormal{(b)}] A \bem random dynamical system (RDS)\eem on the 
space $(E,\cE)$ over the MDS $\left ( \Omega, \F^0,\P,\left ( \theta_t \right)\right)$ with time $\R^+$ is a mapping 
\begin{align*}
\varphi \,\colon \, [0,\infty) \times   E\times\Omega \to E , \hspace{0.5cm} (t,x,\omega)\mapsto \varphi(t,x,\omega)
\end{align*}
with the following properties:
\begin{itemize}
\item[\textnormal{(i)}] \bem Measurability\eem: $\varphi$ is 
$\left ( \B([0,\infty))\otimes \cE \otimes\F^0 ,\cE\right ) $-measurable.
\item[\textnormal{(ii)}] \bem(Perfect) Cocycle property\eem: 
\begin{align*}
\begin{array}{ll}
\varphi(0,\omega)=\textnormal{id}_{E} & \text{for all } \omega \in \Omega\,,\\[0.3cm]
\varphi(t+s,\omega)=\varphi(t,\theta_s\omega)\circ \varphi(s,\omega)& \text{for all } \omega \in 
\Omega \text{ and all } s,t \ge 0.
\end{array}
\end{align*}
\end{itemize}
The RDS $\varphi$ is called \bem (jointly) continuous \eem  if additionally 
\begin{itemize}
\item[\textnormal{(iii)}] the mapping $(t,x)\mapsto \varphi(t,x,\omega)$ is continuous for all $\omega \in \Omega$.
\end{itemize}
\end{enumerate}
\end{defi}
The following definition is due to Crauel and Flandoli, see \cite{CF94}. 
\begin{defi}\label{attractor}
Let $\varphi$ be an RDS on $E$ over  the MDS $(\Omega,\F^0,\left(\theta_t\right)_{t\in \R},\P)$. 
The random set $A(\omega)$ is called an \bem  attractor \eem for $\varphi$ if
\begin{enumerate}
\item[\textnormal{(a)}] $A(\omega)$ is a random element in the metric space of nonempty compact subsets of $E$ equipped 
with the Hausdorff distance. 
\item[\textnormal{(b)}] $A$ is strictly $\varphi$-invariant, that is, for each $t \ge 0$, there exists a set $\Omega_t$ 
of full measure such that 
$\varphi(t,\omega)(A(\omega))=A(\theta_t\omega)$ for  all $\omega \in \Omega_t$.
\item[\textnormal{(c)}] $\lim_{t\to \infty}\sup_{x \in B} d(\varphi(t,\theta_{-t}\omega)(x), A(\omega))=0$  
almost surely for all bounded closed sets $B \subseteq E$.
\end{enumerate}
\end{defi}

\noindent {\bf Remark.}
Attractors as in the previous definition are often called {\em pullback attractors}. If almost sure 
convergence in part (c) of the definition is replaced by convergence in probability, then $A$ is called a 
{\em weak attractor}, see \cite{Gu449}. For a comparison of 
different concepts of a random attractor for one-dimensional diffusions, see \cite{S02}.\\

We will need the following criterion for the existence of an attractor (a much more general result can be found in \cite{CDS08}). 
For simplicity we formulate it only in case 
$E=\R^d$ equipped with the Euclidean metric. Let $\mathrm{B}_r$ be the closed ball with center 0 and radius $r$. 

\begin{propo}\label{criterion}
Let $\varphi: [0,\infty)\times \R^d\times \Omega \to \R^d $  be a continuous cocycle over 
the metric dynamical system $\big( \Omega,\F^0,\P,(\theta_t)_{t \in \R}\big)$. Then the following are equivalent:
\begin{itemize}
\item[\textnormal{(i)}] $\varphi$ has an attractor.
\item[\textnormal{(ii)}]For every  $R>0$
\begin{align*}
\lim_{r \to \infty} \P\big\{\mathrm{B}_R \subseteq \cup_{s=0}^{\infty} \cap_{t \ge s} \varphi^{-1}(t,\mathrm{B}_r,\theta_{-t}\omega) \big\} =1.
\end{align*}
\end{itemize}
\end{propo}
\begin{proof}
If $\varphi$ has an attractor $A$, then for each $\varepsilon>0$ there exists $r>0$ such that $A$ is contained in $\mathrm{B}_{r-1}$ 
with probability at least $1-\varepsilon$. Part (c) of Definition \ref{attractor} therefore implies 
$$
\P\big\{\mathrm{B}_R \subseteq \cup_{s=0}^{\infty} \cap_{t \ge s} \varphi^{-1}(t,\mathrm{B}_r,\theta_{-t}\omega) \big\} \ge 1-\varepsilon
$$
and therefore (ii) follows.

Conversely, (ii) implies the existence of a random {\em absorbing} set which in turn is sufficient for the existence of an attractor 
(for details, see \cite{CF94} or \cite{CDS08}).
\end{proof}

We will show the existence of an attractor for a class of flows $\phi$ satisfying conditions (A1) and (A2).
Since attractors are defined for RDS rather than flows, we have to make sure that $\phi$ generates 
an RDS in an appropriate sense. This is done in the following proposition which is proved in \cite{AS95}. 
Strictly speaking, the set--up in \cite{AS95} is formulated using slightly stronger smoothness assumptions 
on the coefficients of the SDE than in our set--up due to the fact that the authors of \cite{AS95} use the 
Stratonovich rather than It\^o's integral. It is easy to see however that in the It\^o set--up no 
additional smoothness is required for the following proposition to hold.

\begin{propo}\label{sf_rds} Let $\phi$ be the stochastic flow generated via  
SDE (\ref{gen_sde}) satisfying condition $\mathrm{(A1)}$. Then there is an $\R^d-$valued continuous cocycle $\varphi$ 
over some MDS $(\tilde \Omega,\tilde \F,\tilde \P,( \theta_t)_{t\in \R})$ such that 
the distributions of $\big\{\phi_{s,t}(.,.) :  -\infty <s \le t < \infty \big\}$ and 
$\big\{\varphi(t-s,., \theta_s(.)) : -\infty <s \le t < \infty \big\}$ coincide.
\end{propo}
 From now on we shall identify the flow $\phi$ with the associated RDS $\varphi$ in view of the previous proposition. 
In particular, we will check condition (ii) in Proposition \ref{criterion} with $\varphi^{-1}(t,\mathrm{B}_r,\theta_{-t}\omega)$ 
replaced by $\phi_{-t,0}^{-1}(\mathrm{B}_r,\omega)$
and therefore there will be no need to refer to random dynamical systems in the rest of the paper.

\section{Main Result}\label{Mainresult}
In the following we denote a closed ball in $\R^d$ with center $x$ and radius $r$ by $\mathrm{B}(x,r)$ and define 
$\mathrm{B}_r:=\mathrm{B}(0,r)$ as before.

\begin{theo}\label{main}
Assume that $\mathrm{(A1)}$ and $\mathrm{(A2)}$ are satisfied, and define
$$
\beta_0:=\sqrt{2} \sigma_B \left(\lambda (d-1)+ \sigma_L^2 (d-1)^2+
\sqrt{ \sigma_L^4 (d-1)^4 +2\lambda \sigma_L^2 (d-1)^3} \right)^{1/2}.$$
Let $\phi$ be the flow associated to \eqref{gen_sde}.
\begin{itemize}
\item[\textnormal{a)}] If $\phi$ satisfies $\mathrm{(A3^{\beta})}$ for some $\beta<-\beta_0$, then for each $0 \le 
\gamma < -\beta-\beta_0$, we have
$$
\lim_{r \to \infty} \P\left\{\mathrm{B}_{\gamma t} \subseteq \phi_ {-t,0}^{-1}(\mathrm{B}_r)\mbox{ for all } t \ge 0\right\} =1. 
$$
In particular, $\phi$ has a random attractor.
\item[\textnormal{b)}] If $\phi$ satisfies $\mathrm{(A3_{\beta})}$ for some $\beta>\beta_0$, then for each $0 \le 
\gamma < \beta-\beta_0$, we have
$$
\lim_{r \to \infty} \P\left\{\mathrm{B}_{\gamma t} \subseteq \phi_{0,t}(\mathrm{B}_r) \mbox{ for all } t \ge 0\right\} =1. 
$$
\end{itemize}
\end{theo}

\noindent{\bf Remark.}

The same number $\beta_0$ as in the Theorem also appears in upper bounds for the linear growth rate of the diameter of the 
image of a bounded set under a flow: assume (for simplicity) that $\tilde \phi$ is a flow with $b=0$ satisfying (A1) and (A2). 
Then Theorem 2.3 together with Corollary 2.7 and Proposition 2.8 in \cite{S08} show that 
$$
\limsup_{T \to \infty} \left( \sup_{t \in [0,T]} \sup_{x \in \mathrm{B}_1} \frac{1}{T} |\tilde \phi_{0,t}(x)| \right) \le \beta_0, \mbox{ a.s.} 
$$
It seems plausible that adding a drift $b$ to the flow $\tilde \phi$  which satisfies $\mathrm{(A3^{\beta})}$ for some negative 
$\beta$ will reduce the 
linear expansion rate from $\beta_0$ to $\beta_0 + \beta$. In particular, one may expect that linear expansion stops completely as soon 
as $\beta_0 + \beta$ is negative. Part a) of Theorem \ref{main} shows that this is indeed true: in fact we get linear 
contraction of large balls with rate at least $-\beta_0 - \beta$. For further results concerning upper and lower bounds for the 
growth rate of the image of a bounded set under a flow we refer the reader to \cite{CSS99,CSS00,LS01,CS02,SS02,LS03}.
\cite{S08B} contains an example of a flow in the plane which satisfies (A1), (A2) and (A$3^{\beta}$) for some $\beta<0$ 
for which no attractor exists (not even a weak one) which shows that Theorem \ref{main} becomes wrong if $\beta_0$ 
is replaced by 0 when $d \ge 2$. 

Let us briefly consider the case $d=1$. In this case, part a) of Theorem \ref{main} says that an attractor exists whenever 
(A3$^{\beta}$) holds for some $\beta<0$. In fact, one can say more: if the Markov process generated by the SDE admits an invariant 
probability measure which is ergodic in the sense that all transition probabilities converge to it weakly (and for this to 
be true condition (A3$^{\beta}$) does not need to hold), then the associated RDS automatically admits a {\em weak} attractor 
(for this and more general results on {\em monotone} RDS, see \cite{CS04}).  

Note that if $\phi$ is a flow which satisfies the conclusion of part a) of the theorem, then the inverse flow (i.e.~the flow 
run backwards in time) satisfies the conclusion of b) and vice versa (at least if $\phi$ is onto). Therefore, one could just prove 
one of the two statements and then prove the remaining one via time reversal. Unfortunately, the assumptions in both parts do not 
transform accordingly due to the Stratonovich correction term except for cases in which the correction term vanishes 
(which happens for example in case the driving field $M$ is {\em isotropic}).

\section{Proofs}\label{Proofs}

Let us briefly explain the idea of the proof of part b) of Theorem \ref{main} (we will explain the necessary changes 
for part a) later): we will divide the positive time axis into increasingly long intervals $[T_i,T_{i+1}]$ ($T_0=0$) 
and let $R_i$ be an increasing sequence of positive reals. We will provide an upper bound for the probability $q_i$ 
that the image of $\mathrm{B}_{R_i}$ under $\phi_{T_i,T_{i+1}}$ does {\em not} contain $\mathrm{B}_{R_{i+1}}$.
We will show that the $q_i$ are summable in case the $R_i$ and $T_i$ are chosen 
appropriately and then apply a Borel-Cantelli argument.  
This is not quite enough to prove the result: we have to make sure that we can choose the $R_i$ to grow 
sufficiently quickly and we have to ensure that in between successive $T_i$'s, the image of $\mathrm{B}_{R_i}$ 
contains a slightly smaller ball for {\em all} $t \in [T_i,T_{i+1}]$ with high probability.

In order to estimate the probability that the image of $\mathrm{B}_{R_i}$ under $\phi_{T_i,T_{i+1}}$ does not 
contain $\mathrm{B}_{R_{i+1}}$, we will cover the boundary of $\mathrm{B}_{R_i}$ with a large number $N$ of small balls 
with the same radius. We first provide an upper bound for the probability that a single point $x$ with norm $R_i$ will 
be mapped to a point with norm at most $R_{i+1} +1$ under $\phi_{T_i,T_{i+1}}$. This probability will typically be 
very small because the drift tends to push the trajectory away from the origin. We tune $N$ (and the radii of the balls) 
such that both the probability that at least one of the centers of the $N$ balls moves away too slowly
 and the probability that any of the small balls attains a diameter of size 1 before time $T_{i+1}$ are small (i.e.~summable over 
$i$). The required estimates for the growth of the diameter of a small ball under a flow are provided in the 
appendix. 


We start with a  well-known lemma and then proceed with estimates on the one-point motion. We will often write 
$\phi_t$ instead of $\phi_{0,t}$.

\begin{lemma}\label{comp_lemma_2} 
Let $\left (W_t \right )_{t \ge 0}$ be a standard Brownian motion.
Let $W_t^*:=\sup_{s\le t}W_s $ be its running maximum. 
Then for arbitrary $c \ge 0$ and $t >0$ the following bounds hold:
\begin{align*}
\P\big\{W_t\ge c\big\}\le \frac{1}{2} \ee^{-\frac{c^2}{2t}} \hspace{1cm}\text{ and }\hspace{1cm} 
\P \big\{ W_t^*\ge c\big\}\le \ee^{-\frac{c^2}{2t}}\,.
\end{align*}
\end{lemma}

%
\begin{propo}\label{oneptcond}
Let $\phi$ be a stochastic flow satisfying conditions $\mathrm{(A1)}$ and $\mathrm{(A2)}$. 
Let $1 \le \bar R < R$, $S>\bar R$, $T>0$ and $\beta \in \R$. 
\begin{itemize}
\item[\textnormal{a)}] If $\phi$ satisfies $\mathrm{(A3^{\beta})}$, then for each $|x|=R$, we have 
$$
\P\Big\{|\phi_{0T}(x)|\ge S, \inf_{0 \le t \le T} |\phi_{0t}(x)|\ge\bar R  \Big\} \le 
\exp\Big\{-\frac 12 \Big(\Big(-\frac{R-S}{\sigma_B \sqrt{T}} - \frac{\beta^*(\bar R)}{\sigma_B}\sqrt{T}\Big)
^+\Big)^2\Big\},
$$
where $\beta^*(\bar R):=\sup_{|y|\ge \bar R} \{y\cdot b(y)/|y|\}+ (d-1)\sigma_B^2/(2\bar R)$.
\item[\textnormal{b)}] If $\phi$ satisfies $\mathrm{(A3_{\beta})}$, then for each $|x|=S$, we have 
$$
\P\Big\{|\phi_{0T}(x)|\le R, \inf_{0 \le t \le T} |\phi_{0t}(x)|\ge \bar R    \Big\} \le 
\exp\Big\{-\frac 12 \Big(\Big(\frac{\beta_*(\bar R)}{\sigma_B}\sqrt{T} - \frac{R-S}{\sigma_B \sqrt{T}}\Big)
^+\Big)^2\Big\},
$$
where $\beta_*(\bar R):=\inf_{|y|\ge \bar R} \{y\cdot b(y)/|y|\}$.
\end{itemize}
\end{propo}
\begin{proof}
We first show a). Let $h$ be a smooth function from $[0,\infty)$ to $[0,\infty)$ such that $h(y)=y$ for $y \ge 1$, 
$0<h'(y)  \le 1$ for all $y>0$ and $h'(0)=0$ and   
define $\rho_t(x)=h(|\phi_t(x)|)$. Applying It\^o's formula, we get
\begin{align*}
\dd \rho_t(x)=\dd N_t + f(\phi_t(x))\dd t,
\end{align*}
where 
\begin{eqnarray*}
N_t&=& \sum_{i=1}^d \int_0^t h'(|\phi_s(x)|)\frac{\phi^i_s(x)}{|\phi_s(x)|} M^i(\dd s,\phi_s(x)) 
\,\,\,\,\,\,\,\,\text{ and, on $ \{|x|\ge \bar R\} $}, \\
f(x)&=& \frac{x}{|x|}\cdot b(x) + \frac{1}{2|x|}\mathrm{Tr}\,a(x,x)-\frac{1}{2|x|^3}x^Ta(x,x)x \le 
 \frac{x}{|x|}\cdot b(x) + \frac{d-1}{2|x|}\sigma_B^2\\
&\le& \sup_{|y|\ge \Bar R}\frac{y}{|y|}\cdot b(y) + \frac{d-1}{2\bar R}\sigma_B^2 =\beta^*(\bar R).
\end{eqnarray*}
For the quadratic variation of $N$, we have the following bound:
\begin{align*}
\langle N \rangle_t \le \int\limits_0^t \frac{1}{|\phi_s(x)|^2}\phi_s^T(x)a(\phi_s(x),\phi_s(x))\phi_s(x)\,\dd s\le 
\sigma_B^2 t.
\end{align*}
The continuous local martingale  $N$  can be represented (possibly on an 
enriched probability space) in the form  $N_t=\sigma_B W_{\zeta(t)}$, 
where $W$ is a standard Brownian motion and the family of stopping times $\zeta(s)$ satisfies 
$\zeta(s):= \langle \frac{1}{\sigma_B}N \rangle_s \le s $. Let $\tau:=\inf\{t>0:\rho_t(x)<\bar R\}$. 
For $|x|=R$, we get (using an upper index $^*$ to denote the running maximum as before)
\begin{align*}
&\P\Big(\{\rho_T(x)\ge S\}\cap \{\tau \ge T\} \Big)\le \P\big\{|x|+ N^*_{T}+\beta^*(\bar R) T \ge S \big\} 
= \P\big\{ N^*_T \ge S-R-\beta^*(\bar R) T  \big\}\\ 
&\le \P\Big\{ W_1^* \ge -\frac{R-S}{\sigma_B\sqrt{T}}-\frac{\sqrt T}{\sigma_B}
\beta^*(\bar R)  \Big\} \le \exp \Big\{-\frac 12 \Big(\Big(-\frac{R-S}{\sigma_B\sqrt{T}}-\frac{\sqrt T}{\sigma_B}
\beta^*(\bar R)\Big)^+\Big)^2 \Big\},
\end{align*}
where we used Lemma \ref{comp_lemma_2} in the last step. This proves part a).\\

The proof of part b) is analogous to that of a): just interchange $R$ and $S$ and estimate $f$ from below 
by $\beta_*(\bar R)$ on the set $\{|x|\ge \bar R\}$.
\end{proof}

We continue with the proof of part b) of Theorem \ref{main} which is slightly easier than that of part a).

\begin{propo}\label{oneptcond2}
Let $\phi$ be a stochastic flow satisfying conditions $\mathrm{(A1)}$, $\mathrm{(A2)}$ and $\mathrm{(A3_{\beta})}$ 
for some $\beta \in \R$.  
Let $1 \le \bar R<S$, and let
$\beta_*(\bar R)$ be defined as in Proposition \ref{oneptcond}. Then for each $|x|=S$, we have 
$$
\P\left\{ \inf_{t \ge 0}|\phi_{0t}(x)|\le \bar R \right\} \le 
\exp\Big\{-2(S-\bar R)\beta_*(\bar R)\frac{1}{\sigma_B^2}\Big\}. 
$$
\end{propo}

\begin{proof}
It suffices to prove the statement in case $\beta_*(\bar R) > 0$. 
Using the same notation as in the proof of Proposition \ref{oneptcond},
we get for $|x|=S$:
$$
\rho_t(x)= S + \int_0^t f(\phi_{0s}(x))\dd s + N_t \ge S+\beta_*(\bar R)t+ N_t \;\; \mbox{ on } \{t \le \tau\}.
$$
Therefore, using a well-known formula for the law of the supremum of a Wiener process with drift (e.g. \cite{BS96}, p.197),
 we obtain for $T>0$
\begin{eqnarray*}
\P\{\tau < T\} &\le& \P\left\{ \inf_{0\le t \le T}|\phi_{0t}(x)|\le \bar R \right\}\\
&\le& \P\left\{ \inf_{0\le t \le T} \{S+\beta_*(\bar R) t + N_t \}\le \bar R \right\}\\
&\le& \P\left\{ \inf_{0\le t \le T} \{\beta_*(\bar R) t - \sigma_B W_t^*\} \le \bar R -S\right\}\\
&\le& \P\left\{ \sup_{t \ge 0} \Big\{ -\frac{\beta_*(\bar R)}{\sigma_B} t + W_t\Big\} \ge \frac{S-\bar R}{\sigma_B} \right\}\\
&=& \exp \Big\{ -2\beta_*(\bar R) \frac{S-\bar R}{\sigma_B^2} \Big\},
\end{eqnarray*}
so the assertion in the proposition follows.
\end{proof}

The following proposition is a rather easy consequence of the preceding two propositions and the results in the appendix.
\begin{propo}\label{consequence}
Let $\psi: [1,\infty) \to [1,\infty)$ be strictly increasing such that $\lim_{S \to \infty} \frac{\psi(S)}{S}=0$ and 
$\lim_{S \to \infty} \frac{\log S}{\psi(S)}=0$. Let $\phi$ be a flow satisfying $\mathrm{(A1)}$, 
$\mathrm{(A2)}$ and $\mathrm(A3_{\beta})$ for some $\beta>\beta_0$, where $\beta_0$ is as in Theorem \ref{main}.
Let $\eps \in (0,1/2)$ and $\gamma > 0$ satisfy $\gamma + \eps < \beta - \beta_0$. For $S \ge 2$ define $T:=\psi(S)$, 
$\bar R :=(1-\eps) S$ and $R:=S+\gamma \psi(S)$. Denote
$$
p_S:=\P\left\{ \left\{ \mathrm{B}_R \nsubseteq \phi_{0T}(\mathrm{B}_S)\right\} \cup \cup_{0 \le t \le T} \left\{ \mathrm{B}_{\bar R} \nsubseteq \phi_{0t}(\mathrm{B}_{S})\right\} \right\}.
$$
Then $\limsup_{S \to \infty} \frac {1}{\psi(S)} \log p_S <0$.
\end{propo}
\begin{proof}
We can and will assume that $S>\varepsilon^{-1}$. For each $\xi \in (0,S]$, we can cover $\partial \mathrm{B}_S$ by 
$N=N_{\xi} \le c_d \left(\frac {S}{\xi}\right)^{d-1}$ balls with radius $\xi$ 
centered on $\partial \mathrm{B}_S$, where $c_d$ is 
a universal constant which only depends on the dimension $d$. Specifically, we will let $\xi=\exp\{-\Gamma \psi(S)\}$, 
for some $\Gamma \ge 0$. Denote the balls by $M_1,...,M_N$ and their centers by
$x_1,...,x_N$. Using the flow property and Propositions \ref{oneptcond} and \ref{oneptcond2}, we get
\begin{eqnarray*}
p_S&\le& N \max_i \big\{ \P(|\phi_{0T}(x_i)| \le R+1,\inf_{0 \le t \le T} |\phi_{0t}(x_i)| > \bar R+1) + 
\P(\inf_{0 \le t \le T} |\phi_{0t}(x_i)| \le \bar R+1)\\ 
&&+ \P(\sup_{0 \le t \le T} \, \mathrm{diam}\,\phi_{0t}(M_i) \ge 1) \big\}\\
&\le& N(S,\Gamma) \big(A_1(S)+A_2(S)+A_3(S,\Gamma)\big),
\end{eqnarray*}
where
\begin{eqnarray*}
A_1(S)&=&  \exp\Big\{ -\frac 12 \left( \left( \frac{\beta_*(\bar R + 1)}{\sigma_B} \sqrt{T} - \frac{R+1-S}{\sigma_B \sqrt{T}} \right)^+\right)^2\Big\}\\
&=&  \exp\Big\{ -\frac {\psi(S)}{2 \sigma_B^2} \left( \left(\beta_*(\bar R + 1)- \gamma- \frac{1}{\psi(S)}\right)^+\right)^2 \Big\}\\
A_2(S)&=& \exp\Big\{-2\beta_*(\bar R +1)\frac{S-\bar R-1}{\sigma_B^2}\Big\}=\exp\Big\{-2\beta_*(\bar R +1)\frac{\eps S-1}{\sigma_B^2}\Big\} \\
A_3(S,\Gamma)&=&  \max_i \P\big\{\sup_{0 \le t \le T} \, \mathrm{diam}\,\phi_{0t}(M_i) \ge 1 \big\}  \\
N(S,\Gamma)&=&c_d S^{d-1}\exp\{(d-1)\Gamma \psi(S)\}.
\end{eqnarray*}
Clearly, $\lim_{S \to \infty} \frac{1}{\psi(S)} \log(N(S,\Gamma)A_2(S))=-\infty$.
Further, Theorem \ref{ball}, the remark following the proof of Theorem \ref{ball}, and Lemma \ref{two} show that
$$
\limsup_{S \to \infty}\frac{1}{\psi(S)} \log A_3(S,\Gamma)\le -\frac{1}{2\sigma_L^2}(\Gamma-\lambda)^2, 
$$
provided that $d\ge 2$ and $\Gamma \ge \lambda + \sigma_L^2 d$ or $d=1$ and $\Gamma \ge \lambda$. Therefore,
\begin{equation}\label{right}
\limsup_{S \to \infty}\frac{1}{\psi(S)} \log(N(S,\Gamma)A_3(S,\Gamma)) \le (d-1)\Gamma -\frac{1}{2\sigma_L^2}(\Gamma-\lambda)^2, 
\end{equation}
which is negative if $\Gamma$ exceeds the larger of the two roots of the right-hand side of
\eqref{right}, i.e.
$$
\Gamma > \Gamma_0 := \lambda + \sigma_L^2(d-1) + \big( 2\lambda \sigma_L^2(d-1)+\sigma_L^4(d-1)^2 \big)^{1/2}.
$$
Finally, observing that $(d-1)\Gamma_0=\frac{\beta_0^2}{2\sigma_B^2}$, we get
\begin{eqnarray*}
\limsup_{S \to \infty}\frac{1}{\psi(S)} \log(N(S,\Gamma)A_1(S)) &\le& (d-1)\Gamma -\frac{1}{2\sigma_B^2}(\beta-\gamma)^2\\
&\le& (d-1)\Gamma -\frac{1}{2\sigma_B^2}(\beta_0+\eps)^2 <0,\\
\end{eqnarray*}
provided that $\Gamma - \Gamma_0>0$ is sufficiently small. Therefore, we can find some $\Gamma > \Gamma_0$ satisfying all 
of these conditions and the proof is complete.
\end{proof}

Now, we can easily complete the proof of part b) of Theorem \ref{main}.\\

\noindent {\bf Proof of Theorem \ref{main} b).}
Let $0<\gamma < \beta - \beta_0$ and choose $\varepsilon \in (0,1/2)$ such that $\gamma + \varepsilon < \beta - \beta_0$. 
Let $S_0 \ge 2$ and define recursively $S_{i+1}=S_i + \gamma \psi(S_i)$. Define $p_S$ as in the previous proposition. 
By the previous proposition, we know that $\sum_i p_{S_i}$ converges provided that $\sum_i \exp\{-c\psi(S_i)\}$ converges for every 
$c >0$, which is easily seen to be true if we take $\psi(x)=x^{\alpha}$ for some $\alpha \in (0,1)$, and part b) of 
Theorem \ref{main} then follows from the first Borel-Cantelli Lemma and the time-homogeneity of $\phi$.
\hfill $\Box$\\

We now provide the proof of part a) of Theorem \ref{main}. It is partly analogous to the previous one with the exception 
that it does not seem to be obvious how to prove the analog of Proposition \ref{oneptcond2}. 
The following two propositions provide additional estimates for the one-point motion.

\begin{propo}\label{oneptcondmodified}
Let $\phi$ be a flow satisfying conditions $\mathrm{(A1)}$, $\mathrm{(A2)}$ and $\mathrm{(A3^{\beta})}$ for some $\beta<0$ 
and let $V>1$ satisfy  $\beta^*(V)\le 0$.  
Then, for each $S\ge R \ge V,\,x \in \R^d$, we have
$$
\P\{ |\phi_{0T}(x)|\ge S,\,\inf_{0 \le t \le T} |\phi_{0t}(x)|\le R\} \le 2\exp\Big\{-\frac{1}{8T}\Big(\frac{S-R}{\sigma_B}\Big)^2\Big\}.
$$
\end{propo}
\begin{proof}
Define $\rho$, $N$, and $W$ as in the proof of Proposition \ref{oneptcond}. Then
\begin{align*}
&\P\{ |\phi_{0T}(x)|\ge S,\,\inf_{0 \le t \le T} |\phi_{0t}(x)|\le R \}\\ 
&\quad \le \P\{\exists\, 0 \le s< t \le T:  \rho_s(x)=R,\, \rho_t(x)=S,\, \inf_{s \le u \le t}\rho_u(x) \ge R\}\\
&\quad \le \P\{\exists\, 0 \le s< t \le T: \sigma_B W(t)-\sigma_B W(s)+ \beta^*(V)(t-s) \ge  S-R\}\\
&\quad \le \P\Big\{\exists\, 0 \le s< t \le 1:  (W(t)-W(s)) 
\ge \frac{S-R}{\sigma_B \sqrt{T}}\Big\} \\
&\quad \le 2 \P\Big\{ \sup_{0 \le t \le 1}W(t)\ge \frac{S-R}{2\sigma_B\sqrt{T}}\Big\}\\
&\quad \le 2 \exp \Big\{-\frac 1{8T} \Big(\frac {S-R}{\sigma_B}\Big)^2\Big\}
\end{align*}
and the proof of the proposition is complete. 
\end{proof}
We remark that the statement in the previous proposition can be sharpened if we do not drop the drift term in the proof 
(see \cite{S08}, Proposition 2.8) but the statement above meets our demands perfectly. We will also need the following result.

\begin{propo}\label{coro} Let $\phi$ be a flow satisfying conditions $\mathrm{(A1)}$, $\mathrm{(A2)}$ 
and $\mathrm{(A3^{\beta})}$ such that $\beta<0$.  Let $V >1$ be such that $\beta^*(V)\le 0$ and 
$\bar R\ge V$.
Then for each $|x|=\bar R$, $\delta >0$, and $h>0$, we have
$$
\P\{\sup_{0 \le s \le h}|\phi_{0,s}(x)|\ge \bar R + \delta\} \le 3\exp\Big\{-\frac 1 8 \frac{\delta^2}{\sigma_B^2 h}\Big\}.
$$
\end{propo}

\begin{proof}
Let
$$
\tau_1:=\inf\{s>0:|\phi_{0,s}(x)|\ge \bar R + \delta\},\qquad \tau_2:=\inf\{s>0:|\phi_{0,s}(x)|\le V\}.
$$
Then
$$
\P\{\sup_{0\le s \le h}|\phi_{0,s}(x)|\ge \bar R + \delta\}
\le \P\{\tau_1 \le h,\,\tau_2>\tau_1\} + \P\{\tau_1 \le h,\,\tau_2<\tau_1\}.
$$
Arguing like in the proof of Proposition \ref{oneptcond}, we get
\begin{align*}
 \P\{\tau_1 \le h,\,\tau_2>\tau_1\}&=\P\{\sup_{0\le s \le h}|\phi_{0,s}(x)|\ge \bar R + \delta,\,\inf_{0 \le s \le \tau_1}|\phi_{0,s}(x)|> V\} 
\le \exp\Big\{ -\frac 12 \frac {\delta^2}{\sigma_B^2 h} \Big\},
\end{align*}
and, arguing like in the proof of Proposition \ref{oneptcondmodified},
\begin{align*}
\P\{\tau_1 \le h,\,\tau_2<\tau_1\}&=\P\{\sup_{0\le s \le h}|\phi_{0,s}(x)|\ge \bar R + \delta,\,\inf_{0 \le s \le \tau_1}|\phi_{0,s}(x)|\le V \}\\
&\le \P\big\{\exists\, 0 \le s< t \le h:  \rho_s(x)=V,\, \rho_t(x)=\bar R + \delta,\, \inf_{s \le u \le t}\rho_u(x) \ge V\big\}\\
& \le \P\Big\{\exists\, 0 \le s< t \le 1:  W(t)-W(s)=\frac{\bar R + \delta-V}{\sigma_B \sqrt{h}}\Big\}\\
&\le 2 \P\Big\{\sup_{0 \le s \le 1}W(s)\ge \frac{\bar R + \delta-V}{2\sigma_B \sqrt{h}}\Big\}\\
&\le 2 \exp\Big\{ -\frac 12 \frac {\delta^2}{4 \sigma_B^2 h} \Big\},
\end{align*}
where we used Lemma \ref{comp_lemma_2} and the fact that $\bar R \ge V$, so the assertion of the corollary follows.
\end{proof}
\bigskip
\noindent {\bf Proof of Theorem \ref{main} a).} For the reader's convenience, we start by stating all assumptions and notation 
in the proof.

Let $\varepsilon \in (0,1/2)$ and $\gamma >0$ satisfy $\gamma+\varepsilon < -\beta - \beta_0$, $\alpha \in (0,\frac 13)$, 
$\frac{2}{\alpha}>\kappa>\nu>1$, $\psi(x)=x^{\alpha}$, $S:=T^{1/\alpha}$, $R:=S+\gamma T$, 
$\bar R:=(1-\ve)S$, $c:=\frac{6}{\pi^2}$, $\ve_j:=c/j^2$ $(j\in \N)$, 
$\delta:=\delta(j,T):=T^{\kappa/2}\left( \frac 23 \right)^{j/2}$. 

Further, $V>1$ is a fixed number (not depending on $T$) such that $\beta^*(V) \le 0$. Since we are only interested in asymptotic statements as
$T \to \infty$ we can and will assume that $\bar R > V$. Let
\begin{eqnarray*}
p_S&:=&\P\{\mathrm{B}_R \nsubseteq \phi_{0T}^{-1}(\mathrm{B}_S) \mbox { or } \mathrm{B}_{\bar R}\nsubseteq \phi_{sT}^{-1} (\mathrm{B}_S) 
\mbox { for some } s \in [0,T] \}\\
&\le&\P\left\{\bigcup_{|x|=R}\left( \left\{  |\phi_{0T}(x)| \ge S\right\} \cap\left\{ \inf_{0 \le t \le T} |\phi_{0t}(x)| \ge \bar R   \right\} \right)   \right\} 
+ \P\left\{ \sup_{|x|=\bar R}  \sup_{s \in [0,T]}   |\phi_{sT}(x)| \ge S    \right\}   \\
&=:&A_1 + A_2.
\end{eqnarray*}
Once we know that $\limsup_{S \to \infty}\frac {1}{\psi(S)}\log p_S<0$, then Theorem \ref{main} a) will follow just 
like part b).
To estimate $A_1$, we cover $\partial \mathrm{B}_R$ by 
$N \le c_d R^{d-1} \ee^{\Gamma (d-1)T}$ balls of radius $\ee^{-\Gamma T}$ 
centered on  $\partial \mathrm{B}_R$ and we obtain
\begin{eqnarray*}
\limsup_{S \to \infty} \frac{1}{\psi (S)} \log A_1 &\le &\Gamma (d-1) -\left\{ \left( \frac{1}{2 \sigma_B^2}(-\beta-\gamma)^2 \right)
\wedge  \left( \frac{1}{2\sigma_L^2} (\Gamma - \lambda)^2   \right)    \right\}<0 
\end{eqnarray*}
for an appropriate choice of $\Gamma \ge 0$ as in the proof of Proposition \ref{consequence} (using part a) of Proposition \ref{oneptcond} instead of 
part b)).\\

To finish the proof of Theorem \ref{main}, it suffices to prove that 
$$
\limsup_{T \to \infty} \frac 1T \log A_2 = 
\limsup_{T \to \infty} \frac 1T \log \P\Big\{ \sup_{|x|=\bar R}  \sup_{s \in [0,T]}   |\phi_{sT}(x)| \ge T^{1/\alpha} \Big\}
=-\infty.
$$ 
Define $Y_s:= \big(\sup_{|x|= \bar R} |\phi_{sT}(x)|-\bar R\big)^+$, 
$Z_s:= \big(\sup_{|x|= \bar R +T^{1/\alpha}\ve/2} |\phi_{sT}(x)|-\bar R\big)^+$,   $0 \le s \le T$. 
We have 
\begin{align*}
& \hspace{-.5cm}\limsup_{T \to \infty} \frac 1T \log \P\Big\{ \sup_{|x|=\bar R}  \sup_{s \in [0,T]}   |\phi_{sT}(x)| \ge T^{1/\alpha} \Big\}\\
= & \limsup_{T \to \infty}\frac 1T \log \P\big\{ \sup_{s \in [0,T]}  Y_s \ge \ve  T^{1/\alpha}   \big\}\\
\le & \limsup_{T \to \infty}\frac 1T \log\Big((T+1)\max_{0 \le s \le T-1} \P\Big\{\sup_{s \le t \le s+1} Y_t\ge  \ve   T^{1/\alpha}\Big\} \Big)\\
= & \limsup_{T \to \infty}\frac 1T \log \max_{0 \le s \le T-1} \P\Big\{\sup_{s \le t \le s+1} Y_t\ge  \ve   T^{1/\alpha}\Big\}\\
\le & \limsup_{T \to \infty}\frac 1T \log \max_{1 \le s \le T}\Big(\P\Big\{\sup_{s-1 \le t \le s}\sup_{|x|=\bar R}|\phi_{ts}(x)| \ge \bar R+ 
\frac \ve 2 T^{1/\alpha}\Big\} + \P\Big\{ Z_s \ge \ve  T^{1/\alpha}\Big\} \Big).\\
\end{align*}
We treat the two terms in the last sum separately. We start with the second one. We want to show that
\begin{equation}\label{small}
 \limsup_{T \to \infty}\frac 1T \log\Big(\sup_{0 \le s \le T}\P\big\{ Z_s \ge \ve   T^{1/\alpha}\big\}\Big)=-\infty.
\end{equation}
To show this, fix $0 \le s \le T$, abbreviate $\hat R:=\bar R+T^{1/\alpha}\ve /2$ and cover  the boundary $\partial \mathrm{B}_{\hat R}$ by  
$N \le c_d\hat R^{d-1} \ee^{\Gamma (d-1)T}$ balls of radius $\ee^{-\Gamma T}$ centered on $\partial \mathrm{B}_{\hat R}$ 
for some $\Gamma >0$ (the constant $c_d$ can be chosen to depend on $d$ only).  Number the balls by $B_1,...,B_N$ and 
their centers by $x_1,...,x_N$. Then
\begin{align*}
\P&\big\{ Z_s \ge \ve   T^{1/\alpha}\big\}= \P\big\{ \sup_{|x|=\hat R} |\phi_{sT}(x)|\ge \bar R + \ve  T^{1/\alpha}\big\}\\
&\le N\Big( \sup_{|x|=\hat R} \P\big\{|\phi_{sT}(x)|\ge \bar R +  \ve  T^{1/\alpha}-1 \big\} + 
\max_{i=1,...,N} \P\{ \mathrm{diam}\; \phi_{sT}(B_i) \ge 1\}\Big)\\
&\le N\Big( \sup_{|x|=\hat R} \P\big\{|\phi_{sT}(x)|\ge \bar R +  \ve  T^{1/\alpha}-1,\,\inf_{s\le t \le T}|\phi_{st}(x)|\ge V \big\}\\ 
&+ \sup_{|x|=\hat R} \P\big\{|\phi_{sT}(x)|\ge \bar R +  \ve  T^{1/\alpha}-1, \,\inf_{s\le t \le T}|\phi_{st}(x)|\le V \big\}  + 
\max_{i=1,...,N} \P\{ \mathrm{diam}\; \phi_{sT}(B_i) \ge 1\}\Big).
\end{align*}
Estimating the three summands using Propositions \ref{oneptcond}a), \ref{oneptcondmodified}, and Theorem  \ref{ball}, 
respectively, we obtain \eqref{small} by letting $\Gamma \to \infty$ (after taking the $\limsup$ over $T$).\\

It remains to show that
\begin{equation}\label{toshow}
\limsup_{T \to \infty}\frac 1T \log \max_{1 \le s \le T}\P\Big\{\sup_{s-1 \le t \le s}\sup_{|x|=\bar R}|\phi_{ts}(x)| \ge \bar R+ 
\frac \ve 2 T^{1/\alpha}\Big\} = -\infty.
\end{equation}
Proving this is not entirely straightforward. One might try to proceed as (by now) usual by covering 
$\partial \mathrm{B}_{\bar R} \times [s-1,s]$ by small balls and controlling the diameter of their images at time $s$ and 
the norm of the images of their centers at time $s$. One of the obstructions to this approach is that we have no uniform 
control of the component of the drift $b$ towards the origin, i.e. the norm of the solution process can 
drop considerably within a very short time (resulting in an uncontrollable increase of the diameter of a small 
space-time ball within a short time). What we can control is the speed away from the origin thanks to assumption
(A3$^{\beta}$). 
Therefore we proceed as follows: define  
$$
X_t:=\sup_{|x|=\bar R} \big( |\phi_{t1}(x)|-\bar R\big)^+,\qquad t \in [0,1].
$$
Applying Lemma \ref{chain} with $\ve_j:=c/j^2$, $j \in  \N$ (with $c=6/\pi^2$), we obtain
\begin{align*}
&\hspace{-.5cm}\P\Big\{\sup_{s-1 \le t \le s}\sup_{|x|=\bar R}|\phi_{ts}(x)| \ge \bar R+ \frac \ve 2 T^{1/\alpha}\Big\}\\ 
&=\P\Big\{\sup_{0\le t \le 1}\sup_{|x|=\bar R}|\phi_{t1}(x)| \ge \bar R+ \frac \ve 2 T^{1/\alpha}\Big\}\\
&=\P\Big\{\sup_{0\le t \le 1} X_t \ge \frac \ve 2 T^{1/\alpha}\Big\}\\
&\le \sum_{j=1}^\infty 2^{j-1} \sup_{0 \le t \le 1-2^{-j}} \P\Big\{ X_t-X_{t+2^{-j}} \ge \frac{c}{j^2} \frac \ve 2 T^{1/\alpha} \Big\}
\end{align*}
To estimate the probabilities in the last sum, we cover the boundary $\partial \mathrm{B}_{\bar R}$ by  
$N=N_j \le c_d(\bar R 2^j\ee^{T^\nu})^{d-1}$ balls of radius $2^{-j}\ee^{-T^\nu}$ centered on $\partial \mathrm{B}_{\bar R}$, 
where $\nu \in (1,2)$   
(the constant $c_d$ can be chosen to depend on $d$ only).  Number the balls by $M_1,...,M_N$ and 
their centers by $x_1,...,x_N$. For fixed $j,t$ and $|x|=\bar R$, let $\tilde x$ be the projection 
of $\phi_{t,t+2^{-j}}(x)$ on $\partial \mathrm{B}_{\bar R}$ (there will be no need to worry about the possible non-uniqueness of 
$\tilde x$). For $u \ge 0$ we get

\begin{align}\label{4star}
\begin{split}
&\P\Big\{ X_t-X_{t+2^{-j}} \ge u \Big\} \le N\Big(\max_{i=1,...N} \P\{ \mathrm{diam}\, \phi_{t1}(M_i) \ge \frac u2\}\\ 
&\hspace{.2cm}+\sup_{|x|=\bar R} \P\big\{ \big| |\phi_{t1}(x)|\vee\bar R
-|\phi_{t+2^{-j},1}(\tilde x)|\vee \bar R \big|\ge \frac u2,\, |\phi_{t,t+2^{-j}}(x)|
\ge \bar R\big\}\Big) 
\end{split}
\end{align}
There are two terms to estimate. We start with the second one. Recalling that $\delta =T^{\kappa/2}\big(\frac 23\big)^{j/2}$, 
and assuming that $|x|=\bar R$, we have
\begin{align}\label{3star}
\begin{split}
&\hspace{-.5cm} \P\big\{ \big| |\phi_{t1}(x)|\vee\bar R-|\phi_{t+2^{-j},1}(\tilde x)|\vee \bar R\big| 
\ge \frac u2,\, |\phi_{t,t+2^{-j}}(x)|\ge \bar R\big\}\\
&\le  \P\big\{|\phi_{t,t+2^{-j}}(x)|\ge \bar R+\delta\big\} + 
\sup_{|y|=\bar R,|y-z|\le \delta} \P\big\{\big| |\phi_{t+2^{-j},1}(y)|\vee\bar R-|\phi_{t+2^{-j},1}(z)|\vee \bar R\big| \ge \frac u2\big\}.
\end{split}
\end{align}
Again, we have two terms to estimate. By Proposition \ref{coro}, we have
\begin{align*}
\P&\big\{|\phi_{t,t+2^{-j}}(x)|\ge \bar R+\delta\big\} = \P\big\{|\phi_{0,2^{-j}}(x)|\ge \bar R+\delta\big\}
\le 3\exp\Big\{-\frac 18 \frac {T^{\kappa}\big(\frac43\big)^j}{\sigma_B^2}\Big\} 
\end{align*}
and therefore
\begin{align*}
&\hspace{-.5cm} \sum_{j=1}^\infty 2^{j-1} N \P\big\{|\phi_{t,t+2^{-j}}(x)|\ge \bar R+\delta\big\}\\
&\le \frac 32 c_d\bar R^{d-1}\ee^{T^\nu(d-1)}  \sum_{j=1}^\infty 2^{jd} \exp\Big\{-\frac 1{8\sigma_B^2}  
T^{\kappa} \Big(\frac 43\Big)^j\Big\}\\
&\le  \frac 32 c_d\bar R^{d-1}\ee^{T^\nu(d-1)} 2^d\exp\Big\{-\frac 1{8\sigma_B^2}  
T^{\kappa} \frac 43\Big\}\Big(1-2^d\exp\Big\{-\frac 1{8\sigma_B^2}  
T^{\kappa} \frac 43\Big( \frac 43-1\Big)\Big\}\Big)^{-1},
\end{align*}
where we estimated the infinite sum of the form $\sum_{j=1}^{\infty} p_j$ from above by  the geometric series 
$p_1\sum_{j=0}^{\infty} \big(\frac{p_2}{p_1}\big)^j$. Since $\kappa>\nu>1$, the term converges to zero 
superexponentially in $T$.

Next, we estimate the second term in \eqref{3star}. 
Lemma \ref{two} and Lemma \ref{comp_lemma_2} show that for $y,z \in \R^d$ such that $|y-z|\le \delta$ we have
\begin{align}\label{erste}
\begin{split}
&\hspace{-.5cm} \P\big\{\big| |\phi_{t+2^{-j},1}(y)|\vee \bar R-|\phi_{t+2^{-j},1}(z)|\vee \bar R\big| \ge \frac u2\big\}\\
&\le  \P\big\{|\phi_{t+2^{-j},1}(y)-\phi_{t+2^{-j},1}(z)| \ge \frac u2\big\}\\
&\le  \P\big\{\delta \exp\{\sigma_L W_1^* +\lambda\} \ge \frac u2\big\}\\
&\le \exp\Big\{-\frac 1 {2\sigma_L^2}\Big( \Big(\log \frac u{2\delta}-\lambda\Big)^+\Big)^2\Big\}
\end{split}
\end{align}
and for $|y|=\bar R$, $|y-z|\le \delta$, we have
\begin{align}\label{zweite}
\begin{split}
&\hspace{-.5cm} \P\big\{\big| |\phi_{t+2^{-j},1}(y)|\vee\bar R-|\phi_{t+2^{-j},1}(z)|\vee \bar R\big| \ge \frac u2\big\}\\
&\le  2\sup_{|x|\le \bar R + \delta}\P\big\{ (|\phi_{t+2^{-j},1}(x)|-\bar R)^+ \ge \frac u2\big\}.\\ 
\end{split}
\end{align}
We use \eqref{erste} for $j \ge T$ and \eqref{zweite} for $j \le T$ and assume that $T \ge 1$ is so large that $(j \log \frac 32)/4 \ge 
-\log(\frac{c}{j^2} \frac{\ve}{4 \ee^{\lambda}})$ holds for all $j\ge T$. Applying Proposition \ref{coro}, 
we obtain, for $T$ sufficiently large
\begin{align*}
&\hspace{-.5cm}  \sum_{j=1}^\infty 2^{j-1} N  
\sup_{0 \le t \le 1-2^{-j}}\sup_{|y|=\bar R, |y-z|\le \delta}\P\big\{\big| |\phi_{t+2^{-j},1}(y)|\vee \bar R-|\phi_{t+2^{-j},1}(z)|\vee \bar R\big| 
\ge \frac c{j^2} \frac \ve {4}T^{1/\alpha}\big\}\\
&\le  c_d \bar R^{d-1} \ee^{T^\nu(d-1)}
\Big( \sum_{j=1}^{\lfloor T \rfloor}  2^{jd} \sup_{0 \le t \le 1-2^{-j} } \sup_{|x|\le\bar R +\delta}\P\big\{ |\phi_{t+2^{-j},1}(x)| \ge \bar R +
\frac c {j^2} \frac \ve 4 T^{1/\alpha} \big\}       \\
&\hspace{.5cm}+ \sum_{j=\lceil T \rceil}^\infty 2^{jd-1} 
\exp\Big\{-\frac 1 {2\sigma_L^2}\Big( \Big(\log \frac {c\ve}{4j^2\ee^{\lambda}} +\frac{j}{2} \log \frac 32 +\Big(\frac 1 \alpha - 
\frac \kappa 2\Big) \log T\Big)^+\Big)^2\Big\}\Big)\\
&\le  c_d \bar R^{d-1} \ee^{T^\nu(d-1)} \Big(3 \sum_{j=1}^{\lfloor T \rfloor} 2^{jd} \exp\Big\{ -\frac{1}{8\sigma_B^2}
\Big(\Big( \frac{c\ve}{4j^2}T^{1/\alpha}-\delta\Big)^+\Big)^2\Big\}\\
&\hspace{.5cm}+   \sum_{j=\lceil T \rceil}^\infty 2^{jd-1}  \exp\Big\{-\frac 1 {2\sigma_L^2}\frac {j^2}{16} (\log \frac 32)^2\Big\}\Big)\\
&\le  c_d \bar R^{d-1} \ee^{T^\nu(d-1)}\Big( 3 \sum_{j=1}^{\lfloor T \rfloor} 2^{jd} \exp\Big\{ -\frac{1}{8\sigma_B^2}
\Big( \frac{c\ve}{8T^2}T^{1/\alpha}\Big)^2\Big\}            \\
&\hspace{.5cm}+  \sum_{j=\lceil T \rceil}^\infty 2^{jd-1}  \exp\Big\{-\frac 1 {2\sigma_L^2}\frac {jT}{16} (\log \frac 32)^2\Big\}\Big).\\ 
\end{align*}
Evaluating the geometric series and estimating the sum by $T$ times the largest (namely the last) summand, we see that the 
whole expression decays superexponentially in $T$.

Finally, we estimate the first term in \eqref{4star}. Applying Theorem \ref{ball}a) with $q=d+1$ and Lemma \ref{two}, we get
\begin{align*}
&\hspace{-.5cm}  \sum_{j=1}^\infty 2^{j-1} N \P\big\{ \mathrm{diam}\, \phi_{t1}(M_i) \ge \frac c{j^2} \frac \ve 4 T^{1/\alpha}\big\}\\
&\le h_d \bar R^{d-1}\ee^{T^\nu(d-1)} \ee^{-T^\nu(d+1)} \ee^{(\lambda + \frac 12 (d+1)\sigma_L^2)(d+1)T}  
T^{-(d+1)/\alpha} \sum_{j=1}^\infty j^{2(d+1)} 2^{-j(d+1)}  2^{jd-1},  
\end{align*}
which decays to zero 
superexponentially as $T \to \infty$ (here $h_d$ depends on the parameters of the SDE and on $\ve$ but not on $T$).
Therefore, \eqref{toshow} follows and the proof of Theorem \ref{main} is complete. \hfill $\Box$

\section{Appendix}
To prove Theorem \ref{main}, we need the following result. Part b) of the following theorem is also contained in
\cite{S08}. We provide its proof for the reader's convenience (and because it is short). 
 
\begin{theo}\label{ball}
Let $(t,x) \mapsto \phi_t(x)$ be a continuous random field, $(t,x) \in [0,\infty) \times \R^d$
taking values in a separable complete metric space $(E,\rho)$. Assume that there exist numbers $\Lambda \ge0$, 
$\sigma>0$ and $\bar c>0$ such that for each $x,y \in \R^d$, $T>0$, and $q\ge 1$, we have
\begin{equation}\label{H}
\left(\E \sup_{0 \le t \le T} (\rho(\phi_t(x),\phi_t(y)))^q \right)^{1/q} \le \bar c \,|x-y| 
\exp \{(\Lambda + \frac{1}{2}q\sigma^2)T\}.
\end{equation}
\begin{itemize}
\item[a)]
For each cube ${\mathcal X}$ with side length $\xi$, $T>0$, $u>0$, and $\kappa \in (0,1-d/q)$ we have
$$
\P\big\{\sup_{x,y \in {\mathcal X}} \sup_{0 \le t \le T} \rho(\phi_t(x),\phi_t(y))\ge u\big\} \le 
\Big(\frac{2d}{1-2^{-\kappa}}\Big)^q \frac{\bar c^qd2^{q\kappa-q+d}}{1-2^{q\kappa-q+d}} \exp\{(\Lambda + \frac 12 q \sigma^2)qT\}
\xi^q u^{-q}.
$$
\item[b)] 
For $\gamma>0$, define
\begin{eqnarray*}
I(\gamma) :=   \left\{ 
\begin{array}{ll}
\frac{(\gamma-\Lambda)^2}{2\sigma ^2} &{\rm if } \, \gamma \ge \Lambda + \sigma^2 d\\
d (\gamma-\Lambda - \frac{1}{2}\sigma^2 d)  &{\rm if } \,\Lambda + \frac{1}{2}\sigma^2 d \le \gamma \le \Lambda + \sigma^2 d\\
0 &{\rm if } \,\gamma \le \Lambda + \frac{1}{2}\sigma^2 d.
\end{array}
 \right.
\end{eqnarray*}
Then, for each $u>0$, we have
$$
\limsup_{T \to \infty} \frac{1}{T} \sup_{{\mathcal X}_T} \log \P\{\sup_{x,y \in {\mathcal X}_T} 
\sup_{0 \le t \le T} \rho(\phi_t(x),\phi_t(y)) \ge u\} \le -I(\gamma),
$$
where $\sup_{{\mathcal X}_T}$ means that we take the supremum over all cubes ${\mathcal X}_T$ in $\R^d$ 
with side length $\exp\{-\gamma T\}$.
\end{itemize}
\end{theo}

The proof of Theorem \ref{ball} relies on the following (quantitative) version of Kolmogorov's continuity 
theorem which is proved in  \cite{S08}.

\begin{lemma}\label{Kolmo}
Let $\Theta=[0,1]^d$ and assume that there exist $a,b,c >0$ such that, for all $x,y \in [0,1]^d$, we have
$$
{\mathbf E} \left((\hat \rho(Z_x,Z_y))^a\right) \le c |x-y|_1^{d+b}.
$$
Then $Z$ has a continuous modification (which we denote by the same symbol). For each $\kappa \in (0,b/a)$, there exists a 
random variable $S$ such that ${\mathbf E}(S^a) \le \frac{cd 2^{a\kappa-b}}{1- 2^{a\kappa-b}}$ and
$$
\sup \left\{ \hat \rho(Z_x(\omega),Z_y(\omega)): x,y \in [0,1]^d,|x-y|_{\infty} \le r \right\} \le 
\frac{2d}{1-2^{-\kappa}} S(\omega) r^{\kappa}
$$
for each $r \in [0,1]$.
In particular, for all $u > 0$, we have
\begin{equation}\label{estikolmo}
{\mathbf P}\left\{ \sup_{x,y \in [0,1]^d} \hat \rho(Z_x,Z_y) \ge u \right\} \le 
\left( \frac{2d}{1-2^{-\kappa}}\right)^a \frac{cd 2^{a\kappa-b}}{1- 2^{a\kappa-b}}u^{-a}.
\end{equation}
\end{lemma}

\noindent {\bf Proof of Theorem \ref{ball}.} 
Let $T>0$. 
Without loss of generality, we assume that ${\mathcal X}:={\mathcal X}_T=[0,\xi]^d$. 
Define $Z_x(t):=\phi_t(\xi x)$, $x \in {\mathbf R}^d$. 
For $q \ge 1$, \eqref{H} implies
$$
\left({\mathbf E} \sup_{0 \le t \le T} \rho(Z_x(t),Z_y(t))^q\right)^{1/q} \le \bar c \xi |x-y|  
{\mathrm e}^{(\Lambda+\frac{1}{2}q\sigma^2)T},
$$ 
i.e.~ the assumptions of Lemma \ref{Kolmo} are satisfied with $a=q$, $c=\bar c^q \exp\{(\Lambda + \frac{1}{2} q \sigma^2)qT\}\xi^q$ 
and $b=q-d$ for any $q>d$. Therefore we get for $\kappa \in (0,b/a)$:
$$
\P\big\{\sup_{x,y \in {\mathcal X}} \sup_{0 \le t \le T} \rho(\phi_t(x),\phi_t(y))\ge u\big\} \le 
\Big(\frac{2d}{1-2^{-\kappa}}\Big)^q \frac{\bar c^qd2^{q\kappa-q+d}}{1-2^{q\kappa-q+d}} \exp\{(\Lambda + \frac 12 q \sigma^2)qT\}
\xi^q u^{-q},
$$
so part a) follows. Inserting $\xi=\ee^{-\gamma T}$, taking logs, dividing by $T$, letting $T \to \infty$ 
and optimizing over $q>d$ yields part b) of Theorem \ref{ball}. \hfill $\Box$\\

\noindent {\bf Remark.} If, in addition to the assumptions in Theorem  \ref{ball}, the map $x \mapsto \phi_t(x)$ is one-to-one 
for all $t$ and $\omega$, then part b) of Theorem   \ref{ball} holds with $d-1$ replaced by $d$ in the definition of $I(\gamma)$ 
since we can apply Lemma \ref{Kolmo} to each of the faces of ${\mathcal X}_T$ and the supremum over 
$x,y \in {\mathcal {X}}_T$ is attained for $x,y$ on the boundary of ${\mathcal X}_T$.\\ 

The following lemma is almost identical to Lemma 4.1 in \cite{S08} 
and Lemma 5.1 in \cite{CSS00}. We provide its proof, since our assumption (A1) is slightly weaker (in some respect) 
than in those references.

\begin{lemma}\label{two}
Let $\mathrm{(A1)}$ be satisfied. 
Then, for each $x,y \in \R^d$, there exists a Wiener process $W$, such that 
\begin{equation}
\sup_{0 \le t \le T}|\phi_t(x)-\phi_t(y)|\le \ee^{\sigma_L W_T^* + \lambda T}
\end{equation}
for all $T>0$. 
In particular, the assumptions of Theorem \ref{ball} hold with $E=\R^d$, $\rho = |.|$, $\sigma=\sigma_L$, $\Lambda=\lambda$ 
and $\bar c =2$.
\end{lemma}

\noindent {\bf Proof.} Fix $x,\,y \in \R^d$, $x \neq y$ and define

$$D_{t} := \phi_t(x) - \phi_t(y),\;\;Z_{t} := \frac{1}{2} \log(|D_t|^2). $$
Therefore, $Z_t = f (D_t)$ where $f(z) := \frac{1}{2}\log(|z|^2)$. Note that $D_t \neq 0$ for all $t\ge0$ by the one-to-one
property. Using It\^{o}'s formula, we get
\begin{eqnarray*}
   {\mathrm d}  Z_t &=& \frac{D_{t} \cdot \left ( M({\mathrm d}  t, 
  \phi_t (x)) - M({\mathrm d} t,\phi_t (y)) \right )}{| D_{t}|^2 } + 
  \frac{D_t \cdot \left ( b ( \phi_t(x)) - b(\phi_t (y)) \right )}{| D_{t}|^2}  \,{\mathrm d}  t\\
 & & + \frac{1}{2} \frac{1}{|D_{t}|^2} {\mbox {Tr}} \left ( 
  {\cal A} (\phi_t(x), \phi_t(y)) \right )  {\mathrm d}  t - \sum_{i,j} \frac{D^i _{t} D^j _{t}}{(|D_{t}|^2)^2} 
  {\cal A}_{i,j} (\phi_t(x), \phi_t(y))  \,{\mathrm d}  t.
\end{eqnarray*}
We define the local martingale $N_{t}, t \ge 
0$ by
$$
N_t = \int^t_{0}\frac{D_{s}}{|D_{s}|^2} \cdot \left ( M({\mathrm d}  s, \phi_s(x)) - M({\mathrm d}  s, 
\phi_s(y))\right )
$$
and obtain 
$$
Z_t = Z_{0} + N_t + \int^t_{0} \alpha (s,\omega)  \,{\mathrm d}  s,  
$$
where
$$
\sup_{x,y}\sup_s {\mathrm {esssup}}_{\omega}  
|\alpha (s,\omega)|  \le \lambda
$$ 
and
\begin{equation}\label{kappa}
{\mathrm d}  \langle N \rangle_{t} = \sum_{i,j} 
\frac{D^i_{t}D^j_{t}}{(|D_{t}|^2)^2}{\cal{A}}_{i,j}(\phi_t (x), \phi_t (y))  \,{\mathrm d}  t \le \sigma_L^2   \,{\mathrm d} t. 
\end{equation}
Since $N$ is a continuous local martingale with 
$N_{0} = 0$, there exists a standard Brownian motion $W$ (possibly on an enlarged probability space) 
such that $N_t=\sigma_L W_{\tau(t)}$, $t \ge 0$ and (\ref{kappa}) 
implies $\tau(t) \le t$ for all $t \ge 0$. Hence
\begin{equation}
Z_{t} \le \log |x-y| + \sigma_L \; \sup_{0 \le s \le t} W_{s} + \lambda t. 
\end{equation}
Exponentiating the last inequality completes the proof of the lemma. \hfill $\Box$\\

In the proof of part a) of Theorem \ref{main} we need the following {\em one-sided} Chaining Lemma (without absolute values). 

\begin{lemma}\label{chain}
 Let $T>0$ and let $X_t,\,0 \le t \le T$ be a real-valued process with right-continuous paths. Let $\varepsilon_j,\,j \in \N$ 
be positive and satisfy $\sum_{j=1}^{\infty} \varepsilon_j =1$. Then
$$
\P \big\{ \sup_{0 \le t \le T} X_t - X_T \ge u\big\} \le \sum_{j=1}^{\infty}  2^{j-1} 
\sup_{0 \le s < t \le T,t-s = 2^{-j} T} \P \big\{  X_s - X_t \ge \varepsilon_j u\big\},\,u \ge 0. 
$$
\end{lemma}

\begin{proof}
For $t \in (0,T]$ and $j \in \N$ define $\alpha(j,t):=\lceil 2^{j-1} t/T\rceil 2^{-j+1} T$. Then, by right-continuity, 
$$
X_t-X_T = \sum_{j=1}^{\infty} \big( X_{\alpha(j+1,t)} - X_{\alpha(j,t)} \big)
$$
for $t>0$, and therefore
$$
\P \big\{ \sup_{0 \le t \le T} X_t - X_T \ge u\big\} \le \sum_{j=1}^{\infty} 2^{j-1} \max_{k=1,...,2^{j-1}} 
\P \big\{ X_{(k-\frac 12)2^{-j+1} T} - X_{k2^{-j+1} T} \ge \varepsilon_j u \big\} 
$$
proving the lemma.
\end{proof}

\bibliography{biblio}
\bibliographystyle{abbrv}

\end{document}